\documentclass[12pt]{article}

\usepackage{amssymb,amsmath,amsthm,mathrsfs,mathtools}
\usepackage[left=1in, right=1in]{geometry}
\usepackage{appendix}
\usepackage[round]{natbib}
\newcommand{\eps}{\varepsilon}
\newcommand{\supp}{\mathrm{supp}}
\def\R{{\mathbb R}}
\def\N{{\mathbb N}}
\def\Z{{\mathbb Z}}
\def\t{{\mathbb T}}
\def\T{{\mathbb T}}
\def\D{{\mathcal D}}

\def\e{{\rm e}}
\def\d{{\,\rm d}}
\def\ri{{\rm i}}

\def\w{{\rm w}}

\def\[{\left\lfloor}
\def\]{\right\rceil}
\def\<{\langle}
\def\>{\rangle}
\def\:{{\colon}}
\def\S{{\mathcal S}}

\def\be#1{\begin{equation}\label{#1}}
\def\ee{\end{equation}}

\newtheorem{theorem}{Theorem}
\newtheorem{lemma}[theorem]{Lemma}

\newtheorem{corollary}[theorem]{Corollary}

\newtheorem{definition}[theorem]{Definition}

\begin{document}
\title{Energy conservation in the 3D Euler equations on $\T^2\times\R_+$}
\author{James C.\ Robinson \and Jos\'e L.\ Rodrigo \and Jack W.D.\ Skipper}
\maketitle

\begin{abstract}
The aim of this paper is to prove energy conservation for the incompressible Euler equations in a domain with boundary. We work in the domain $\T^2\times\R_+$, where the boundary is both flat and has finite measure.

However, first we study the equations on domains without boundary (the whole space $\R^3$, the torus $\mathbb{T}^3$, and the hybrid space $\T^2\times\R$). We make use of some of the arguments of Duchon \& Robert ({\it Nonlinearity} {\bf 13} (2000) 249--255)
to prove energy conservation under the
assumption that $u\in L^3(0,T;L^3(\R^3))$ and one of the two integral conditions
\begin{equation*}
\lim_{|y|\to 0}\frac{1}{|y|}\int^T_0\int_{\R^3} |u(x+y)-u(x)|^3\d x\d t=0
\end{equation*}
or
\begin{equation*}
  \int_0^T\int_{\R^3}\int_{\R^3}\frac{|u(x)-u(y)|^3}{|x-y|^{4+\delta}}\,\d x\,\d y<\infty,\qquad\delta>0,
\end{equation*}
the second of which is equivalent to requiring $u\in L^3(0,T;W^{\alpha,3}(\R^3))$ for some $\alpha>1/3$.

We then use the first of these two conditions to prove energy conservation for a weak solution $u$ on $D_+:=\t^2\times \R_+$: we extend $u$ a solution defined on the whole of $\T^2\times\R$ and then use the condition on this domain to prove energy conservation for a weak solution $u\in L^3(0,T;L^3(D_+))$ that satisfies %one of the two the purely spatial conditions
\begin{equation*}
\lim_{|y|\to 0} \frac{1}{|y|}\int^{T}_{0}\iint_{\t^2}\int^\infty_{|y|}|u(t,x+y)-u(t,x)|^3\d x_3\d x_1\d x_2\d t=0,
\end{equation*}
and certain continuity conditions near the boundary $\partial D_+=\{x_3=0\}$.
% or
% \begin{equation*}
%   \int_0^T\int_{D_+}\int_{D_+}\frac{|u(x)-u(y)|^3}{|x-y|^{4+\delta}}\,\d x\,\d y<\infty,
% \end{equation*}
% which is equivalent to requiring $u\in L^3(0,T;W^{\alpha,3}(D_+))$ for some $\alpha>1/3$.
\end{abstract}

\section{Introduction}

Energy conservation  for  solutions of the incompressible  Euler equations
$$
\partial_tu+(u\cdot\nabla)u+\nabla p=0\qquad\nabla\cdot u=0
$$
 has long  been a topic of interest. While for sufficiently smooth solutions $u$ a standard integration-by-parts argument shows that energy is conserved ($\|u(t)\|_{L^2}=\|u(0)\|_{L^2}$ for every $t\ge0$) for weak solutions $u\in L^\infty(0,T;L^2)\cap L^3(0,T;L^3)$ we do not have the regularity needed to perform these operations. \cite{onsager1949statistical} conjectured that weak solutions to the Euler equations satisfying a H\"older continuity condition of order greater than one third should conserve energy.

The study of energy conservation for this system has so far been carried out on domains without boundary, either the whole space $\R^3$ or the torus $\mathbb{T}^3$. In this paper we aim to treat the question on the domain $\T^2\times\R_+$, which involves a flat boundary with finite measure.

The first proof of energy conservation for weak solutions was given by \cite{eyink1994energy} on the torus, assuming that the solution satisfies $u(\cdot,t)\in C^\alpha_\star$ for $\alpha>1/3$ with a uniform bound for $t\in [0,T]$. A definition of the space $C^\alpha_\star$ equivalent to that of Eyink's is as follows: expand $u$ as the Fourier series
$$
u=\sum_{k\in\Z^3}\hat u_k\e^{\ri k\cdot x},
$$
imposing conditions to ensure that $u$ is real ($\hat u_k=\overline{\hat u_{-k}}$) and is divergence free ($k\cdot\hat u_k=0$); then $u\in C^\alpha_\star(\T^3)$ if
$$
\sum_{k\in\Z^3}|k|^\alpha|\hat u_k|<\infty.
$$
%. For each $k\in\Z^3$ choose $e_1(k)$ and $e_2(k)$ such that the set $(e_1(k),e_2(k),\hat k)$ (where $\hat k=k/|k|$) forms an orthonormal basis of $\R^3$; then $\{e_1(k)\e^{\ri k\cdot x},e_2(k)\e^{\ri k\cdot x}\}_k$ forms an orthogonal basis for the collection of divergence-free functions in $L^2(\t^3$). Expanding $u$ in terms of this basis,
%$$
%u=\sum_{k=-\infty}^\infty \sum_{n=1,2}A_{kn} e_n (k){\rm{e}}^{{\rm{i}}k\cdot x},
%$$
%we say that $u\in C^\alpha_\star(\t^3)$ if
%%
%%for a function $u$ on $\t^3$ there exists set $\{A_{kn}(u)\}$ of Fourier coefficients with $k\in\Z^d$ and $\{e_n\}^d_{n=1}$ are a set of orthonormal basis vectors of $\R^d$; we obtain a Fourier series expansion on $\t^d$, given by
%%\begin{equation}
%%\sum_{k,n}A_{kn} e_n (k){\rm{e}}^{{\rm{i}}k\cdot x}.
%%\end{equation}
%%Then a function is an element of $C^\alpha_\star(\t^3)$ if
%$$
%\sum_{k=-\infty}^\infty \sum_{n=1,2}|k|^\alpha|A_{kn}|< \infty.
%$$
Requiring $u\in C^\alpha_\star$ with $\alpha>1/3$ is a stronger condition than the one-third H\"older continuity conjectured by Onsager.

Subsequently \cite{constantin1994onsager}
gave a short proof of energy conservation, in the framework of Besov spaces (but still on the torus), under the weaker assumption that
\begin{equation}\label{CET}
u\in L^3(0,T;B^\alpha_{3,\infty})\qquad\mbox{with}\qquad\alpha>1/3.
\end{equation}
As $C^\alpha\subset B^\alpha_{3,\infty}$ this proves Onsager's Conjecture. Here $B^s_{p,r}$ denotes a Besov space as defined in \cite{bahouri2011fourier} and \cite{lemarie2002recent}. %defined in Appendix \ref{besovappendix}.

\cite{duchon2000inertial} showed that solutions satisfying a weaker regularity condition still conserve energy. They derived a local energy equation that contains a term $D(u)$ representing the dissipation or production of energy caused by the lack of smoothness of $u$; this term can be seen as a local version of Onsager's original statistically averaged description of energy dissipation. They showed %, by integrating over the domain, that if $\|D(u)\|_{L^1(0,T;L^1(\R^3))}=0$ then energy is conserved. Since this paper takes their theory as a starting point, we discuss their argument in more detail in Section \ref{section3}.  Further, \cite{duchon2000inertial} show
that if $u$ satisfies
\begin{equation}\label{RDcondition}
 \int |u(t,x+\xi)-u(t,x)|^3 \d x\le C(t)|\xi|\sigma(|\xi|),
\end{equation}
where $\sigma(a)\to0$ as $a\to 0$ and $C\in L^1(0,T)$, then
$\|D(u)\|_{L^1(0,T,L^1(\T^3))}=0$ and hence the kinetic energy is conserved. The condition in \eqref{RDcondition} is weaker than \eqref{CET}. A detailed review examining this and further work relating to Onsager's conjecture is given by \cite{eyink2006onsager}.

More recently energy conservation was shown by \cite{cheskidov2008energy} when $u$ lies in the space $L^3(0,T;B^{1/3}_{3,c(\N)})$, where $B^{1/3}_{3,c(\N)}$ is a subspace of $B^{1/3}_{3,\infty}$. % see Appendix \ref{besovappendix}
In fact \cite{cheskidov2008energy} showed that energy conservation holds for solutions satisfying the still weaker condition
\begin{equation*}%\label{CCRScondition}
\lim_{q\to\infty}\int^T_0 2^q\|\Delta_q u\|^3_{L^3}\d t=0,
\end{equation*}
where $\Delta_q$ performs a smooth restriction of $u$ into Fourier modes of order $2^q$.
In a follow-up paper \cite{RS09} (see also \citealp{shvydkoy2010lectures}) states that this condition is equivalent to
\begin{equation}\label{Scondition}
\lim_{|y|\to 0}\frac{1}{|y|}\int^T_0\int |u(x+y)-u(x)|^3\d x\d t=0,
\end{equation}
and proves a local energy balance under this condition. We observe that condition \eqref{RDcondition} has similar form to \eqref{Scondition}, yet explicitly separates the limit and the integrability in time. This makes \eqref{Scondition} less restrictive.

In this paper we use an approach similar to that of \cite{RS09}, but rather than basing our argument on the approach of  \cite{constantin1994onsager} we adopt some of the ideas from \cite{duchon2000inertial} and give a direct proof that energy conservation follows on the whole domain (this simplifies matters since the pressure no longer plays a role) under the condition that
$$
\int_{\R^3}\int_{\R^3}\nabla\varphi_\eps (\xi)\cdot (v(x+\xi)-
 v(x))|v(x+\xi)-v(x)|^2\d \xi \d x\to0
 $$
 as $\eps\to0$, where $\varphi$ is an even mollifier.

 Given this condition it is relatively simple to show energy conservation under the assumption \eqref{Scondition}, which we do in Theorem \ref{conservation}, and under the alternative condition
 \begin{equation}\label{newone}
  \int_0^T\int_{\R^3}\int_{\R^3}\frac{|u(x)-u(y)|^3}{|x-y|^{4+\delta}}\d x\d y<\infty,\qquad\delta>0,
 \end{equation}
 which is equivalent to requiring $u\in L^3(0,T;W^{\alpha,3}(\R^3))$ for some $\alpha>1/3$ (Theorem \ref{X2}).

   For the most significant new contribution of this paper we use condition (\ref{Scondition}) to analyse energy conservation in the domain $D_+:=\T^2\times\R_+$. We show that if $(u,p)$ is a weak solution on $D_+$ then $(u_R,p)$ is a weak solution on $D_-$, where $u_R$ is an appropriately `reflected' version of $u$, and that $u+u_R$ is a weak solution on $D:=\T^2\times\R$. It follows that energy is conserved for $u_E$ under condition \eqref{Scondition}; from here we deduce energy conservation for $u$ under the condition
\begin{equation*}
\lim_{|y|\to 0} \frac{1}{|y|}\int^{t_2}_{t_1}\iint_{\t^2}\int^\infty_{|y|}|u(t,x+y)-u(t,x)|^3\d x_3\d x_1\d x_2\d t=0,
\end{equation*}
   and additional assumptions near $\partial D_+$: we assume that $u$ is continuous at $\t^2\times \{0\}$, for almost every $t$ and $u\in L^3(0,T;L^\infty(\t^2 \times [0,\delta))$ for some $\delta>0$, see Theorem \ref{MainD+}.

\section{Energy conservation without boundaries}\label{section3}

In this first section we treat the incompressible Euler equations on a domain without boundaries: $\R^3$, $\T^3$, or one of the hybrid domains $\T\times\R^2$ or $\T^2\times\R$. We write $D$ in what follows to denote any one these domains, being careful to highlight any differences required in the definitions/arguments required to deal with the periodic or hybrid cases.

\subsection{Weak solutions of the Euler equations}\label{sec:WSE1}

For vector-valued functions $f,g$ and matrix-valued functions $F,G$ we use the notation
$$
 \langle f,g \rangle=\int_{D}f_i(x)g_i(x) \d x \quad \mathrm{and}  \quad \langle F:G \rangle=\int_{D}F_{ij}(x)G_{ij}(x) \d x,
$$
employing Einstein's summation convention (sum over repeated indices).

We use the notation $\D(D)$ to denote the collection of $C^\infty$ functions with compact support in $D$, and $\S(D)$ for the collection of all $C^\infty$ functions with Schwartz-like decay in the unbounded directions of $D$, e.g.\ for $\T^2\times\R$ we require
\begin{equation*}
 \sup_{x\in\T^2\times\R}|\partial^\alpha \phi||x_3|^k \le \infty,
\end{equation*}
for all $\alpha,k\ge 0$ where $\alpha$ is a multi-index over all the spatial variables $(x_1,x_2,x_3)$. Note that in periodic directions the requirement of `compact support' is trivially satisfied. The spaces $\D_\sigma(D)$ and $\S_\sigma(D)$ consist of all divergence-free elements of the $\D(D)$ or $\S(D)$.

We denote by $H_\sigma(D)$ the closure of $\D_\sigma(D)$ in the norm of $L^2(D)$; this coincides with the closure of $\S_\sigma(D)$ in the same norm.

Elements of $H_\sigma(D)$ are divergence free in the sense of distributions, i.e.
 \begin{equation}\label{incomp}
 \langle u,\nabla \phi \rangle=0\qquad\mbox{for all}\quad\phi \in {\mathcal D}(D);
 \end{equation}
 but in fact this equality holds for all $\phi\in\S(D)$, and even for all $\phi\in H^1(D)$: indeed, since $\S_\sigma(D)$ is dense in $H_\sigma(D)$, for any $u\in H_\sigma(D)$ we can find $(u_n)\in\S_\sigma(D)$ such that $u_n\to u$ in $H^1(D)$, and then for any $\phi\in H^1(D)$ we have
 $$
 \<u,\nabla\phi\>=\lim_{n\to\infty}\<u_n,\nabla\phi\>=\lim_{n\to\infty}\<\nabla\cdot u_n,\phi\>=0
 $$
 (cf.\ Lemma 2.11 in Robinson et al., 2016, for example).

 In a slight abuse of notation we denote by $C_\w([0,T];H_\sigma)$ the collection of all functions $u\:[0,T]\to H_\sigma(D)$ that are
%\begin{equation}
%\mathcal{H}_\sigma:= \{u\in C_w([0,T];H_\sigma) \}
%\end{equation}
% which denotes functions that are
weakly continuous into $L^2$, i.e.
\begin{equation}\label{whatCwis}
 t\mapsto \langle u(t), \phi \rangle
\end{equation}
is continuous for every $\phi\in L^2(D)$. Note that $C_\w([0,T];H_\sigma)\subset L^\infty(0,T;H_\sigma)$.

We take as our space-time test functions the elements of
 %\begin{equation*}
%\mathcal{D_\sigma}:= \{\psi\in C_c^\infty(\t^3\times [0,T];\R^3)\colon \nabla \cdot \psi(x,t)=0\quad \mbox{for all } (x,t)\in\T^3\times [0,T] \}.
%\end{equation*}
%Since the compact-support property is not preserved under the Helmholtz decomposition, it is also useful to consider a larger space of test functions
$$
{\mathcal S}_\sigma^T:=\{\psi\in C^\infty(D\times[0,T]):\ \psi(\cdot,t)\in\S_\sigma(D) \mbox{ for all } t\in [0,T] \}.
$$
We choose these functions to take values in $\S_\sigma$ (rather than in $\D_\sigma$) since the property of compact support is not preserved by the Helmholtz decomoposition, whereas such a decomposition respects Schwartz-like decay.

\begin{lemma}\label{HD21}
  Any $\psi \in \mathcal S$ can be decomposed as $\psi=\phi +\nabla \chi $, where $\phi\in\S_\sigma$ and $\chi\in\S$,  and %hence for any $v\in H_\sigma$ we have
%\begin{equation}\label{HD21}
% \langle v,\psi \rangle= \langle v,\phi \rangle +\langle v,\nabla \chi  \rangle= \langle v,\phi \rangle.
%\end{equation}
%We also have,
\begin{equation}\label{HDSob}
\|\psi\|_{H^s}+\|\nabla\chi\|_{H^s}\le C_s\|\psi\|_{H^s}
\end{equation}
for each $s\ge0$.
\end{lemma}

\begin{proof} (Cf.\ Theorem 2.6 and Exercise 5.2 in \citealp{RRS}.)
  Since $\psi\in\mathcal S$ we can write $\psi$ as a hybrid Fourier series/inverse Fourier transform, using Fourier series in the periodic directions and the Fourier transform in the unbounded directions. For example, in the case $D=\T^2\times\R$ we have
  $$
  \psi(x)=\int_{-\infty}^\infty\sum_{(k_1,k_2)\in\Z^2}\hat u(k)\e^{\ri k\cdot x}\,\d k_3,
  $$
  and we can set
  $$
  \phi(x)=\int_{-\infty}^\infty\sum_{(k_1,k_2)\in\Z^2}\left(I-\frac{k\otimes k}{|k|^2}\right)\hat u(k)\e^{\ri k\cdot x}\,\d k_3,
  $$
  and
  $$
  \chi(x)=\int_{-\infty}^\infty\sum_{(k_1,k_2)\in\Z^2}\frac{k\cdot\hat u(k)}{|k|^2}\e^{\ri k\cdot x}\,\d k_3;
  $$
  in the fully periodic case we omit the $k\otimes k/|k|^2$ term when $k=0$.   It is easy to check that these functions have the stated properties.
\end{proof}

Assuming that $u$ is a smooth solution of the Euler equations
$$
\partial_tu+(u\cdot\nabla)u+\nabla p=0\qquad\nabla\cdot u=0
$$
if we multiply by an element of $\S_\sigma^T$ and integrate by parts in space and time then we obtain (\ref{Weaksolution}) below; the pressure term vanishes since there are no boundaries and $\psi$ is divergence free. Requiring only (\ref{Weaksolution}) to hold we obtain our definition of a weak solution.

\begin{definition}[Weak Solution]\label{def:weakonU3}
We say that $u\in C_\w([0,T];H_\sigma)$ is a weak solution of the Euler equations on $[0,T]$, arising from the initial condition $u(0)\in H_\sigma$, if
\begin{equation*}
 \langle u(t),\psi(t)\rangle-\langle u(0),\psi(0)\rangle-\int^t_0\langle u(\tau),\partial_t\psi(\tau)\rangle \d \tau
 =\int^t_0\langle u(\tau)\otimes u(\tau):\nabla\psi(\tau)\rangle\d \tau\label{Weaksolution}
\end{equation*}
for every $t\in[0,T]$ and any $\psi\in \mathcal{S}_\sigma^T$.
\end{definition}

We note here that replacing $\S_\sigma^T$ by $\D_\sigma^T$ leads to an equivalent definition (via a simpler version of the argument of Lemma \ref{generalpsi}, below).

Throughout the paper we let $\varphi$ be an even scalar function in $C_c^\infty(B(0,1))$  with $\int_{\R^3} \varphi=1$ and for any $\eps>0$ we set $\varphi_\eps(x)=\eps^{-3}\varphi(x/\eps)$. Then for any function $f$ we define the mollification of $f$ as $f_\eps:= \varphi_\eps \star f $ where $\star$ denotes convolution. Thus
$$
 f_\eps(x)= \varphi_\eps \star f(x):=\int_{\R^3}\varphi_\eps(x-y)f(y)\d y= \int_{B(0,\eps)}\varphi_\eps(y)f(x-y)\d y.
$$
In the periodic directions we extend $f$ by periodicity in this integration. %To ensure that this operation is symmetric, it is important that $\varphi$ is even.
%It will be important that the fact that we take $\varphi$ to be even ensures that the operation of mollification has the symmetry property
We insist that $\varphi$ is even since this ensures that the operation of mollification satisfies the `symmetry property'
 \begin{equation}\label{symmetry}
 \langle \varphi_\eps \star u, v\rangle=\langle  u,\varphi_\eps \star v\rangle.
 \end{equation}

Our aim in the next section is to show the validity of the following two equalities that follow from the definition of a weak solution in \eqref{Weaksolution}. The first is
\begin{equation}\label{uepstestfunction}
 \langle u(t),u_\eps(t)\rangle-\langle u(0),u_\eps(0)\rangle-\int^t_0\langle u(\tau),\partial_t u_\eps(\tau)\rangle \d \tau=\int^t_0\langle u(\tau)\otimes u(\tau):\nabla u_\eps(\tau)\rangle\d \tau;
\end{equation}
this amounts to using $u_\eps$, a mollification of the solution $u$, as a test function in \eqref{Weaksolution}: we need to show that there is sufficient time regularity to do this, which we do in Section \ref{WSM}. The second is
\begin{equation}\label{mollfiedeqn}
 \int^t_0\langle\partial_t u_\eps(\tau), u(\tau)\rangle \d \tau=-\int^t_0\langle \nabla \cdot [u(\tau)\otimes u(\tau)]_\eps,u(\tau)\rangle\d \tau.
\end{equation}
One could see this heuristically as a ``mollification of the  equation'' tested with $u$; we will show that this can be done in a rigorous way in Section \ref{MoE}.
We can then add these equations and take the limit as $\eps \to 0$ to obtain the equation for conservation (or otherwise) of energy (Section \ref{Econ}).

\subsection{Using $u_\eps$ as a test function}\label{WSM}

We will show that if $u$ is a weak solution then in fact \eqref{Weaksolution} holds for a larger class of test functions with less time regularity. We denote by $C^{0,1}([0,T];H_\sigma)$ the space of Lipschitz functions from $[0,T]$ into $H_\sigma$.

\begin{lemma}\label{generalpsi}
 If $u$ is a weak solution of the Euler equations in the sense of Definition \ref{def:weakonU3} then \eqref{Weaksolution} holds for every $\psi\in\mathcal{L}_\sigma$, where
 \begin{equation*}
  \mathcal{L}_\sigma:=L^1(0,T;H^3)\cap C^{0,1}([0,T];H_\sigma).
 \end{equation*}
\end{lemma}

\begin{proof}
 For a fixed $u$ we can write \eqref{Weaksolution} as $E(\psi)=0$ for every $\psi\in\S_\sigma^T$, where
 \begin{align*}
  E(\psi):=\langle u(t),\psi(t)\rangle-\langle u(0),\psi(0)\rangle-\int^t_0\langle u(\tau),&\partial_t\psi(\tau)\rangle \d \tau\\
  &-\int^t_0\langle u(\tau)\otimes u(\tau):\nabla\psi(\tau)\rangle\d \tau.
 \end{align*}
Since $E$ is linear in $\psi$, and  $\S_\sigma^T$ is dense in $\mathcal{L}_\sigma$ with respect to the norm
\begin{equation*}
 \|\psi\|_{L^1(0,T;H^3)}+\|\psi\|_{C^{0,1}([0,T];L^2)},
\end{equation*}
%since functions in $\mathcal{D}_\sigma$ are dense in $L^1(0,T;H^3) \cap C^{0,1}([0,T];L^2)$,
to complete the proof it suffices to show that $\psi\mapsto E(\psi)$ is bounded in this norm. We proceed term-by-term:
$$
 \left|\langle u(t),\psi(t)\rangle-\langle u(0),\psi(0)\rangle\right|\le 2\|u\|_{L^\infty(0,T;L^2)}\|\psi\|_{L^\infty(0,T;L^{2})},
 $$
 using the fact that $u\in C_{\rm w}([0,T];H_\sigma)$;
 \begin{align*}
 \left|\int^t_0\langle u(\tau),\partial_\tau\psi(\tau)\rangle \d \tau\right|&\le T\|u\|_{L^\infty(0,T;L^2)}\|\psi\|_{C^{0,1}([0,T];L^{2})};\qquad\mbox{and}\\
 \left|\int^t_0\langle u(\tau)\otimes u(\tau):\nabla\psi(\tau)\rangle\d \tau\right|&\le T\|u\|^2_{L^\infty(0,T;L^2)}\|\nabla\psi\|_{L^{1}(0,T;L^{\infty})}.
\end{align*}
It follows that
\begin{equation*}
|E(\psi)|\le C \|u\|_{L^\infty(0,T;L^2)}\|\psi\|_{C^{0,1}([0,T];L^{2})} +C\|u\|^2_{L^\infty(0,T;L^2)}\|\psi\|_{L^{1}(0,T;H^3)}
\end{equation*}
and so we obtain the desired result.
\end{proof}

We now study the time regularity of $u$ when paired with a sufficiently smooth function that is not necessarily divergence free.

\begin{lemma}\label{uLipcitz}
 If $u$ is a weak solution then
 \begin{equation}\label{Lipschitzintime}
  |\langle u(t)-u(s),\phi\rangle|\le C|t-s|\quad\mbox{for all}\quad \phi\in H^3(\R^3),
 \end{equation}
where $C$ depends only on $\|u\|_{L^\infty(0,T;L^2)}$ and $\|\phi\|_{H^3}$.
\end{lemma}

\begin{proof}
We use Lemma \ref{HD21} to decompose $\phi\in \S(D)$ as $\phi=\eta+\nabla \sigma$, where $\eta\in\S_\sigma(D)$, $\sigma\in\S(D)$, and
\begin{equation*}
 \|\nabla \eta\|_{L^\infty}\le \|\nabla \eta \|_{H^2} \le \|\eta\|_{H^3} \le C\|\phi\|_{H^3},
\end{equation*}
using (\ref{HDSob}) and the fact that %the Leray projector (the map $\phi\mapsto \eta$) is bounded in $H^s$ for any $s\ge 0$ (see, for example, Chapter 2 and 3 of \cite{lions1997mathematical} or Chapter 2 of \cite{RRS}) and that
$H^2(D)\subset L^\infty(D)$. Since $u(t)$ is incompressible for every $t\in [0,T]$, we have
\begin{equation*}
 \langle u(t)-u(s),\phi \rangle= \langle u(t)-u(s),\eta+\nabla\sigma \rangle=\langle u(t)-u(s),\eta \rangle.
\end{equation*}
Since $\eta\in \S_\sigma$ and $\partial _t\eta=0$ it follows from the definition of a weak solution at times $t$ and $s$ that
\begin{equation*}
 \langle u(t)-u(s),\phi \rangle =\int^t_s\langle u(\tau)\otimes u(\tau):\nabla \eta \rangle \d \tau
\end{equation*}
and hence
$$
 |\langle u(t)-u(s),\phi \rangle|\le \|u\|^2_{L^\infty(0,T;L^2)}\|\nabla \eta\|_{L^\infty}|t-s|\le C\|u\|^2_{L^\infty(0,T;L^2)}\|\phi\|_{H^3}|t-s|,
$$
which gives \eqref{Lipschitzintime} for all $\phi\in\S$. Now simply observe that $\S(D)$ is dense in $H^3(D)$ to obtain \eqref{Lipschitzintime} as stated.
\end{proof}

A striking corollary of this weak continuity in time is that a mollification \emph{in space only} yields a function that is Lipschitz continuous \emph{in time}.

\begin{corollary}
We have $u_\eps\in{\mathcal L}_\sigma$ for any $\eps>0$; in particular the function $u_\eps(x,\cdot)$ is Lipschitz continuous in $t$ as a function into $L^2(D)$:
%, uniformly for $x\in\t^3$:
\begin{equation}\label{LipschitzL2}%\label{Lipschitzsupremum}
\|u_\eps(\cdot,t)-u_\eps(\cdot,s)\|_{L^2}\le C_\eps\|u\|_{L^\infty(0,T;L^2)}^2|t-s|.
% \|u_\eps(\cdot,t)-u_\eps(\cdot,s)\|_{L^\infty}\le C_\eps|t-s|.
\end{equation}
\end{corollary}

\begin{proof}
Take $f\in L^2(D)$ with $\|f\|_{L^2(D)}=1$, and let $\phi=f_\eps$. Then $\phi\in H^3(D)$, and
using the symmetry property (\ref{symmetry}) we have
\begin{align*}
\<u(t)-u(s),\phi\>&=\<u(t)-u(s),f_\eps\>\\
&=\<(u_\eps(t)-u_\eps(s)),f\>.
\end{align*}
Since we have $\|\phi\|_{H^3}\le C_\eps\|f\|_{L^2}=C_\eps$ it follows from Lemma \ref{uLipcitz} that
$$
|\<u_\eps(t)-u_\eps(s),f\>|\le C_\eps\|u\|_{L^\infty(0,T;L^2)}^2|t-s|.
$$
Since this holds for every $f\in L^2(D)$ with $\|f\|_{L^2(D)}=1$ we obtain the inequality \eqref{LipschitzL2} and $u_\eps\in C^{0,1}([0,T];L^2)$.

%For each $t,s\in[0,T]$ let $\phi=(u_\eps(t)-u_\eps(s))_\eps$. Then $\phi\in H^3(D)$, and using the symmetry property (\ref{symmetry}) we have
%\begin{align*}
%\<u(t)-u(s),\phi\>&=\<u(t)-u(s),(u_\eps(t)-u_\eps(s))_\eps\>\\
%&=\<(u_\eps(t)-u_\eps(s)),(u_\eps(t)-u_\eps(s))\>=\|u_\eps(t)-u_\eps(s)\|_{L^2}^2.
%\end{align*}
%Since for each choice of $t,s\in[0,T]$ we have $\|\phi\|_{H^3}\le C_\eps\|u\|_{L^\infty(0,T;L^2)}$ the inequality \eqref{LipschitzL2} follows and $u_\eps\in C^{0,1}([0,T];L^2)$.

%
% For each $x\in \t^3$ and each $i=1,2,3$ take $\phi_i(\cdot)=e_i\varphi_\eps(x-\cdot)$, where $e_i$ is the $i$th coordinate vector, and note that $\|\phi_i\|_{H^3}$ is independent of $x$ and $i$. Then
% \begin{equation*}
%  [u_\eps]_i(x,t)=\int \langle u(y,t),\phi_i(x-y)\rangle\d y=\langle u(t),\phi_i\rangle,
% \end{equation*}
%and the inequality in \eqref{Lipschitzsupremum} follows from Lemma \ref{uLipcitz}.

As mollification commutes with differentiation it follows that $u_\eps$ is divergence free. %Further we have shown that $u_\eps\in C^{0,1}([0,T];L^\infty)$ and as
%\begin{equation*}
%\|u_\eps\|_{C^{0,1}([0,T];L^2)}\le C \|u_\eps\|_{C^{0,1}([0,T];L^\infty)},
%\end{equation*}
% we see that
 Finally, since $u\in L^\infty(0,T; L^2)$, we observe that $u_\eps \in L^\infty(0,T;H^3)$ and
\begin{equation*}
 \|u_\eps\|_{L^1(0,T;H^3)} \le T \|u_\eps\|_{L^\infty(0,T;H^3)}
\end{equation*}
as $[0,T]$ is bounded.
 \end{proof}

Since $u_\eps\in \mathcal{L}_\sigma$  it follows from Lemma \ref{generalpsi} that we can use $u_\eps$ as a test function the definition of a weak solution and obtain
$$
 \langle u(t),u_\eps(t)\rangle-\langle u(0),u_\eps(0)\rangle-\int^t_0\langle u(\tau),\partial_t u_\eps(\tau)\rangle \d \tau=\int^t_0\langle u(\tau)\otimes u(\tau):\nabla u_\eps(\tau)\rangle\d \tau;
$$
we have validated equation \eqref{uepstestfunction}, the first of the two equalities we need.

\subsection{`Mollifying the equation'}\label{MoE}

We will now derive \eqref{mollfiedeqn}. The trick is to test with a mollified test function and move the mollification from the test function onto the terms involving $u$; all terms are then smooth enough to allow for an integration by parts. %We will then show that we can generalise to the test function so that we can use the weak solution to %test against.

\begin{lemma}
 If $u$ is a weak solution then
 \begin{equation}\label{Mollifiedeqn}
  \int^t_0\langle\partial_t u_\eps,\phi\rangle \d \tau =-\int ^t_0\langle \nabla \cdot [u\otimes u]_\eps ,\phi\rangle \d \tau
 \end{equation}
for every $t\in[0,T]$ and any $\phi\in \mathcal{S}_\sigma^T$.
\end{lemma}
\begin{proof}
 Take $\phi\in\mathcal{S}_\sigma^T$, and use $\psi:=\varphi_\eps \star \phi$ as the test function in the weak formulation \eqref{Weaksolution}. Then
 \begin{multline*}
  \langle u(t),(\varphi_\eps\star\phi)(t)\rangle-\langle u(0), (\varphi_\eps\star\phi)(0)\rangle-\int^t_0\langle u(\tau),\partial_t[\varphi_\eps \star \phi](\tau)\rangle \d \tau\\=\int^t_0\langle u(\tau)\otimes u(\tau):\nabla[\varphi_\eps \star \phi](\tau)\rangle\d \tau.
 \end{multline*}
Since the fact that we have chosen $\varphi$ to be even implies that $\langle\varphi_\eps\star u,v\rangle=\langle u,\varphi_\eps\star v\rangle$ (see \eqref{symmetry}) we can move the derivatives and mollification onto the terms involving $u$. We will do this in detail for the term on the right-hand side, since it is the most complicated; the other terms follow similarly. We obtain
\begin{multline*}
 \int^t_0\langle u(\tau)\otimes u(\tau):\nabla[\varphi_\eps \star \phi](\tau)\rangle\d \tau=\int^t_0\langle u(\tau)\otimes u(\tau):\varphi_\eps \star\nabla \phi(\tau)\rangle\d \tau \\=\int^t_0\langle [u(\tau)\otimes u(\tau)]_\eps:\nabla \phi(\tau)\rangle\d \tau=\int^t_0\langle \nabla \cdot[u(\tau)\otimes u(\tau)]_\eps,\phi(\tau)\rangle\d \tau.
\end{multline*}
This implies that
 \begin{align*}
  \langle u_\eps(t), \phi(t)\rangle-\langle u_\eps(0), \phi(0)\rangle-\int^t_0\langle u_\eps(\tau)&,\partial_t \phi(\tau)\rangle \d \tau\\
  &=\int^t_0\langle \nabla \cdot [u(\tau)\otimes u(\tau)]_\eps:\phi(\tau)\rangle\d \tau.
 \end{align*}
 Since $u_\eps$ and $\phi$ are both absolutely continuous in time, the integration-by-parts formula
 \begin{equation*}
  \langle u_\eps(t), \phi(t)\rangle-\langle u_\eps(0), \phi(0)\rangle-\int^t_0\langle u_\eps(\tau),\partial_t \phi(\tau)\rangle \d \tau=\int^t_0\langle \partial_t u_\eps(\tau), \phi(\tau)\rangle \d \tau
 \end{equation*}
 finishes the proof.
\end{proof}

 We now show that \eqref{Mollifiedeqn} holds for a much larger class of functions than $\phi\in \mathcal{S}_\sigma^T$.

\begin{lemma}\label{eqnmollify}
 If $u$ is a weak solution and in addition $u\in L^3(0,T;L^3)$  then \eqref{Mollifiedeqn} holds for any $\phi\in L^3(0,T;L^3) \cap C_\w(0,T;H_\sigma)$.
\end{lemma}

\noindent(Recall that we use $C_\w(0,T;H_\sigma)$ to denote $H_\sigma$-valued functions that are weakly continuous into $L^2$.)

\begin{proof}
First we will obtain from \eqref{Mollifiedeqn} an equation that holds for all test functions $\psi$ from the space $\S(D\times [0,T])$, not just for $\psi\in\mathcal{S}_\sigma^T$. For this we will use the Leray projection $\mathbb{P}$, the projection onto divergence-free vector fields.
% Trivially, for any $\psi \in \mathcal{D}$ we have
% \begin{equation}
%  \int^t_0\langle\partial_t u_\eps+\nabla \cdot [u\otimes u]_\eps ,\psi\rangle \d \tau=\int^t_0\langle\partial_t u_\eps+\nabla \cdot [u\otimes u]_\eps ,\mathbb{P}\psi+(\psi-\mathbb{P}\psi)\rangle \d \tau.
% \end{equation}
Since for any $\psi\in \S(D\times [0,T])$ we have $\mathbb{P}\psi\in \S_\sigma$, it follows from \eqref{Mollifiedeqn} that
%this part tested against the equation goes to zero. Thus we get,
\begin{equation*}
\int^t_0\langle\partial_t u_\eps+\nabla \cdot [u\otimes u]_\eps ,\mathbb{P}\psi\rangle \d \tau=0.
\end{equation*}
Since $\mathbb{P}$ is symmetric ($\langle \mathbb{P}g, f\rangle =\langle g,\mathbb{P} f\rangle $) and $\mathbb{P} \partial_t u_\eps=\partial_t u_\eps$ (since $\mathbb P$ commutes with derivatives and $u_\eps$ is incompressible) we obtain
% while since $(\psi-\mathbb{P}\psi)=\nabla\Phi$ for some $\Phi\in\mathcal D$ and $u_\eps$ is incompressible it follows (using \eqref{incomp}) that
% \begin{equation}
% \int^t_0\langle \partial_tu_\eps ,(\psi-\mathbb{P}\psi)\rangle \d \tau=0.
% \end{equation}
% Therefore
\begin{equation*}
 %\int^t_0\langle\partial_t u_\eps,\psi\rangle +\langle \nabla \cdot [u\otimes u]_\eps ,\mathbb{P}\psi\rangle \d \tau
 \int^t_0\langle\partial_t u_\eps+\mathbb{P}( \nabla \cdot [u\otimes u]_\eps ),\psi\rangle \d \tau=0\qquad\mbox{for every}\quad\psi\in\S(D\times[0,T]).
\end{equation*}
Since $u_\eps$ is Lipschitz in time (as a function from $[0,T]$ into $H_\sigma$) its time derivative $\partial_tu_\eps$ exists almost everywhere (see Theorem 5.5.4 in \cite{AK}, for example) and is integrable; we can therefore deduce using the Fundamental Lemma of the Calculus of Variations ($u\in L^2(\Omega)$ with $\int_\Omega u\cdot\psi=0$ for all $\psi\in C_c^\infty(\Omega)$ implies that $u=0$ almost everywhere in $\Omega$, see e.g.\ Lemma 3.2.3 in \cite{JLJ}) that for almost every $(x,t)\in D\times[0,T]$
\begin{equation*}
\partial_tu_\eps+\mathbb{P}(\nabla \cdot (u\otimes u)_\eps)=0.
\end{equation*}

Observing that $\mathbb{P}\nabla\cdot(u\otimes u)_\eps\in L^{3/2}(0,T;L^{3/2})$  and  that $\partial_tu_\eps$ has the same integrability since $\partial_tu_\eps=-\mathbb{P}\nabla\cdot(u\otimes u)_\eps$, we can now  multiply this equality by any choice of function $\phi\in L^3(0,T;L^3) \cap C_\w(0,T;H_\sigma)$ and integrate:
\begin{align*}
\int^t_0\langle\partial_t u_\eps,\phi\rangle\d\tau&=-\int^t_0\langle \mathbb{P}\nabla \cdot [u\otimes u]_\eps ,\phi\rangle \d \tau\\
&=-\int^t_0\langle \nabla \cdot [u\otimes u]_\eps ,\mathbb{P}\phi\rangle \d \tau=-\int^t_0\langle \nabla \cdot [u\otimes u]_\eps ,\phi\rangle \d \tau,
\end{align*}
%As $\partial_t u_\eps \in L^\infty(0,T;L^2)$ by the equality $\mathbb{P}(\nabla \cdot (u\otimes u)_\eps) \in L^\infty(0,T;L^2)$. We see that $\phi$ is a valid test function, as in $L^3(0,T;L^3)$, thus for any $v\in L^\infty(0,T;L^2)$ we have
%\begin{equation}
% \left|\int^t_0\langle v,\phi \rangle\d t\right|\le \|v\|_{L^\infty(0,T;L^2)}\|\phi\|_{L^1(0,T;L^2)}\le C\|v\|_{L^\infty(0,T;L^2)}\|\phi\|_{L^3(0,T;L^3)}.
%\end{equation}
where we have used the fact that $\mathbb{P}\phi=\phi$ since $\phi(t)\in H_\sigma$ for every $t\in[0,T]$.
%Since now $\nabla \cdot \phi=0$ in $\mathcal{D}'(\t^3)$ for almost every $t$ the projection onto divergence-free vector fields can be removed and \eqref{Mollifiedeqn} follows.
\end{proof}

Note that the condition on $u\in L^3(0,T;L^3)$ is stronger than necessary for the proof but since Theorem \ref{conservation} will need this condition the above result will suffice for our purposes.

We can now use $u$ as a test function in \eqref{Mollifiedeqn} and thereby obtain equation \eqref{mollfiedeqn}, the second of the equalities we need.

\subsection{Energy Conservation}\label{Econ}

We  can now add equations \eqref{uepstestfunction} and \eqref{mollfiedeqn}
to obtain
\begin{multline}\label{energyequationpreconvergence}
 \langle u(t),u_\eps (t) \rangle- \langle u(0),u_\eps (0) \rangle=\\ \int^t_0 \langle u(\tau)\otimes u(\tau) :\nabla u_\eps (\tau) \rangle-\langle\nabla \cdot [u(\tau)\otimes u(\tau)]_\eps ,u(\tau) \rangle\d \tau.
\end{multline}

In order to proceed we will need the following identity. We note that its validity is entirely independent of the Euler equations, but relies crucially on the fact that $\varphi$ is even.

\begin{lemma}
 Suppose that $v\in L^3\cap H_\sigma$ and define
\begin{equation*}
 J_\eps (v):=\int_{\R^3}\int_{\R^3}\nabla\varphi_\eps (\xi)\cdot (v(x+\xi)-
 v(x))|v(x+\xi)-v(x)|^2\d \xi \d x.
\end{equation*}
Then
 \begin{equation*}%\label{Jepsidenty}
  \tfrac{1}{2}J_\eps (v)=\langle\nabla \cdot [v(\tau)\otimes v(\tau)]_\eps ,v(\tau) \rangle-\langle v(\tau)\otimes v(\tau) :\nabla v_\eps (\tau) \rangle.
\end{equation*}
\end{lemma}

\begin{proof}
 We have
 \begin{equation*}
  J_\eps (v)=\int_{\R^3}\int_{\R^3}\partial_i \varphi_\eps (\xi) (v_i(x+\xi)-v_i(x))(v_j(x+\xi)-v_j(x))(v_j(x+\xi)-v_j(x))\d \xi \d x.
 \end{equation*}
Expanding the expression for $J_\eps(v)$ yields
\begin{align*}
&\int_{\R^3}\left\{\int_{\R^3}\partial_i\varphi_{\eps}(\xi)v_i(x+\xi)v_j(x+\xi)v_j(x+\xi)\d \xi-\int_{\R^3} \partial_i\varphi_{\eps}(\xi)v_i(x)v_j(x)v_j(x)\d \xi\right.  \\ %\label{Line2}
&+\int_{\R^3}\partial_i\varphi_{\eps}(\xi)v_i(x+\xi)v_j(x)v_j(x)\d \xi-\int_{\R^3}\partial_i\varphi_{\eps}(\xi)v_i(x)v_j(x+\xi)v_j(x+\xi)\d \xi\\%\label{Line3}
&\left.+2\left[ \int_{\R^3}\partial_i\varphi_{\eps}v_i(x)v_j(x+\xi)v_j(x)\d \xi -\int_{\R^3}\partial_i\varphi_{\eps}v_i(x+\xi)v_j(x+\xi)v_j(x)\d \xi\right]\right\} \d x.
\end{align*}
Note that the second term is zero since $\varphi_\eps$ has compact support, and the third term is zero since $v$ is incompressible.
For the fourth term we can change variable and set $\eta=x+\xi$ to obtain
\begin{equation*}
 -\int_{\R^3}\int_{\R^3}\partial_{\eta_i}\varphi_{\eps}(\eta-x)v_i(x)v_j(\eta)v_j(\eta)\d \eta \d x.
\end{equation*}
As $\partial_i\varphi_{\eps}$ is an odd function we have
\begin{equation*}
 \int_{\R^3}\int_{\R^3}\partial_i\varphi_{\eps}(x-\eta)v_i(x)v_j(\eta)v_j(\eta)\d \eta \d x,
\end{equation*}
which becomes
\begin{equation*}
 \int_{\R^3}v_i(x)\partial_{x_i}\left[\int_{\R^3}\varphi_{\eps}(x-\eta)v_j(\eta)v_j(\eta)\d \eta \right ]\d x=\int_{\R^3}v_i(x)\partial_{i}([v_j v_j]_\eps)\d x=0,
\end{equation*}
where again the term becomes zero as we use the incompressibility of $v$.
A similar calculation for the first term gives
\begin{equation*}
\int_{\R^3}\int_{\R^3}\partial_i\varphi_{\eps}(\xi)v_i(x+\xi)v_j(x+\xi)v_j(x+\xi)\d \xi\,\d x= \int_{\R^3}\partial_i([v_i v_jv_j]_\eps) (x)\, \d x=0,
\end{equation*}
using periodicity.
For the final two terms similar calculations yield
$$
 2\int_{\R^3}[v_j\partial_i(v_jv_j)_\eps -v_jv_i\partial_i(v_j)_\eps]\d x=2[\langle\nabla \cdot [v\otimes v]_\eps ,v \rangle-\langle v\otimes v :\nabla v_\eps \rangle]
$$
and the result follows.
\end{proof}

Note that here again the assumption that $v\in L^3 $ is stronger than needed but will hold when we use the result in Theorem \ref{conservation}.

%We have related the right- hand side of \eqref{energyequationpreconvergence} to a single term that has the spacial derivative shared over the solution cubed.
We now want to   look at the limit as $\eps \to 0$ and see what condition on the solution is needed for the right hand side of \eqref{energyequationpreconvergence} to converge to zero.

Let $t_1,t_2\in [0,T]$ with $t_1<t_2$. We can set $t=t_1$  in \eqref{energyequationpreconvergence} and also set $t=t_2$ in \eqref{energyequationpreconvergence} and then take the difference of these two equations to obtain   %and then up to $t_2$ with $t_2>t_1$ then subtracting these two equations we have,
\begin{align*}
 \langle u(t_2),u_\eps (t_2) \rangle&- \langle u(t_1),u_\eps (t_1) \rangle\\
 &=\int^{t_2}_{t_1} \langle u(\tau)\otimes u(\tau) :\nabla u_\eps (\tau) \rangle-\langle\nabla \cdot [u(\tau)\otimes u(\tau)]_\eps ,u(\tau) \rangle\d \tau\\
 &=-\frac{1}{2}\int^{t_2}_{t_1}J_\eps (u) \d t .
\end{align*}
Therefore talking the limit as $\eps \to 0$, since  $u\in C_\w([0,T];H_\sigma)$ we obtain
\begin{equation*}
\|u(t_2)\|_{L^2}-\|u(t_1)\|_{L^2}=-\frac{1}{2}\lim_{\eps\to 0} \int^{t_2}_{t_1}J_\eps (u) \d t.
\end{equation*}
Hence any condition on $u$ that guarantees that
\begin{equation}\label{DRC}
\lim_{\eps\to 0} \int^{t_2}_{t_1}J_\eps (u) \d t\to0\qquad\mbox{as}\qquad\eps\to0
\end{equation}
ensures energy conservation. We give two such conditions in the next section.

\section{Two spatial conditions for energy conservation in the absence of boundaries}

First we provide another proof (cf.\ \citealp{RS09}) of energy conservation under condition \eqref{Scondition}.

\begin{theorem}\label{conservation}
If $u\in L^3(0,T;L^3(D))$ is a weak solution of the Euler equations that satisfies
\begin{equation}\label{ours}
 \lim_{|y|\to 0} \frac{1}{|y|}\int_{0}^{T}\int_{D}|u(t,x+y)-u(t,x)|^3\d x\d t=0
\end{equation}
then  energy is conserved on $[0,T]$. %The same result holds on the torus under a similar condition.
\end{theorem}

\begin{proof}
We take $t_1,t_2$ with $0\le t_1\le t_2\le T$, and consider the integral of $|J_\eps(u)|$ over $[t_1,t_2]$; our aim is to show that this is zero in the limit as $\eps\to0$. We start by noticing that
 \begin{equation*}
 \int^{t_2}_{t_1}|J_\eps (u)|\d t
 \le\int^{t_2}_{t_1}\int_{D}\int_{\R^3}\frac{1}{\varepsilon^{4}}\left|\nabla\varphi\left(\frac{\xi}{\varepsilon}\right)\right||u(x+\xi)-u(x)|^3\d \xi\d x\d t.
\end{equation*}
We can then change variables $\xi=\eta \varepsilon$ and obtain,
% \begin{equation}
%  \int^{t_2}_{t_1}\int_{\t^3}\int_{\R^3}\frac{1}{\varepsilon^{4}}\left|\nabla\varphi\left(\eta \right)\right||u(x+\varepsilon\eta)-u(x)|^3\varepsilon^3\d \eta\d x\d t
% \end{equation}
\begin{equation*}
 \int^{t_2}_{t_1}|J_\eps (u)|\d t
 \le\int^{t_2}_{t_1}\int_{D}\int_{\R^3}\frac{1}{\varepsilon}\left|\nabla\varphi\left(\eta \right)\right||u(x+\varepsilon\eta)-u(x)|^3\d \eta\d x\d t.
\end{equation*}
Using Fubini's Theorem we can exchange the order of the integrals:
\begin{equation*}%\label{proofdiverges}
 \int^{t_2}_{t_1}|J_\eps (u)|\d t
 \le\int_{\R^3}\int^{t_2}_{t_1}\int_{D}\frac{|u(x+\varepsilon\eta)-u(x)|^3}{|\varepsilon\eta|}\d x\d t\,|\eta|\left|\nabla\varphi\left(\eta \right)\right|\d \eta.
\end{equation*}
Taking limits as $\varepsilon$ goes to zero
\begin{equation*}
 \lim_{\eps\to0}\int^{t_2}_{t_1}|J_\eps(u)|\d t
\le \lim_{\eps\to0}\int_{\R^3}\int^{t_2}_{t_1}\int_{D}\frac{|u(x+\varepsilon\eta)-u(x)|^3}{|\varepsilon\eta|}\d x\d t\,|\eta|\left|\nabla\varphi\left(\eta \right)\right|\d \eta.
\end{equation*}
We are finished if we can exchange the outer integral and limit. This can be done using the Dominated Convergence Theorem.
To do this we define the non-negative function,
\begin{equation*}
 f(y)=\frac{1}{|y|}\int^{t_2}_{t_1}\int_{D}|u(x+y)-u(x)|^3\d x\d t.
\end{equation*}
By assumption $\lim_{|y|\to 0}f(y)=0$, thus for any $\eps>0$, we have  $\sup_{y\in B_0(\eps)}f(y)\le K$ for some $K=K(\eps)$. Further, $\supp(\varphi)$ is compact. Combining these facts we obtain a dominating integrable function
\begin{equation*}                                                                                                                                                                                                                                                                                                                                                                                                                                                                               g(\eta):=K|\eta|\left|\nabla\varphi\left(\eta \right)\right|,
\end{equation*}
and the result follows.\end{proof}

We now show how the general condition in (\ref{DRC}) allows for a simple proof of energy conservation when $u\in L^3(0,T;W^{\alpha,3}(\R^3))$ for any $\alpha>1/3$. The use of condition (\ref{SSC}) to characterise this space is due independently to Aronszajn, Gagliardo, and Slobodeckij, see \cite{DiNPV}, for example.

\begin{theorem}\label{X2}
If $u$ is a weak solution of the Euler equations on the whole space that satisfies $u\in L^3(0,T;W^{\alpha,3}(\R^3))$ for some $\alpha>1/3$, i.e.\ if $u\in L^3(0,T;L^3(\R^3))$ and
  \begin{equation}\label{SSC}
  \int_{\R^3}\int_{\R^3}\frac{|u(x)-u(y)|^3}{|x-y|^{3+3\alpha}}\,\d x\,\d y<\infty,
  \end{equation}
  then energy is conserved. %The same result holds on $\T^3$ provided that one interprets the expression $|x-y|$ in the denominator as the distance on the torus, i.e.
%  $$
%  |x-y|=d_{\T^3}(x,y)=\min_{k\in\Z^3}|x-(y+2\pi k)|.
%  $$
\end{theorem}

\begin{proof}
  First observe that for $\alpha>1/3$ the space $W^{\alpha,3}$ has a factor $|x-y|^{4+\delta}$ in the denominator of (\ref{SSC}), where $\delta=3\alpha-1>0$.

  As in the previous proof, our starting point is that
   \begin{equation*}
 \int^{t_2}_{t_1}|J_\eps (u)|\d t
 \le\int^{t_2}_{t_1}\int_{\R^3}\int_{\R^3}\frac{1}{\varepsilon^{4}}\left|\nabla\varphi\left(\frac{\xi}{\varepsilon}\right)\right||u(x+\xi)-u(x)|^3\d \xi\d x\d t.
\end{equation*}
We can write
\begin{align*}
\int^{t_2}_{t_1}\int_{\R^3}\int_{\R^3}&\frac{1}{\varepsilon^{4}}\left|\nabla\varphi\left(\frac{y-x}{\varepsilon}\right)\right||u(y)-u(x)|^3\d \xi\d x\d t\\
&=\int^{t_2}_{t_1}\int_{\R^3}\int_{\R^3}\frac{1}{\varepsilon^{4}}\left|\nabla\varphi\left(\frac{y-x}{\varepsilon}\right)\right||u(y)-u(x)|^3\d \xi\d x\d t\\
&=\int^{t_2}_{t_1}\int_{\R^3}\int_{\R^3}\frac{|y-x|^{4+\delta}}{\eps^4}\left|\nabla\varphi\left(\frac{y-x}{\varepsilon}\right)\right|\frac{|u(y)-u(x)|^3}{|y-x|^{4+\delta}}\d \xi\d x\d t\\
&\le cK_\varphi \epsilon^\delta\int_{t_1}^{t_2}\int\int \frac{|u(y)-u(x)|^3}{|y-x|^{4+\delta}}\d y\d x\d t=c\eps^\delta,
\end{align*}
since $\|\nabla\varphi\|_{L^\infty}\le K_\varphi$ and the integrand is only non-zero within the support of $\varphi$, i.e.\ where $|y-x|\le2\eps$. Energy conservation now follows.
\end{proof}

\section{Energy Balance on $\t^2\times \R_+$}

%Define the domain $D_+:=\t^2\times \R_+$, then
In this section we first give a definition of what it means to have a weak solution on the domain $D_+:=\t^2\times\R_+$, where $\R_+=[0,\infty)$, that is suitable for our purposes. We then show that such a solution $u$ can be extended to a weak solution $u_E$ on the boundary-free domain $D:=\T^2\times\R$. (Note that in this section we reserve $D$ for this particular domain.)
%will introduce a simple extension operator for a weak solution $u$ on $D_+:=\t^2\times \R_+$ so that it becomes a weak solution $u_E$ on $D$. %We can then find the conditions needed for energy conservation of $u_E$.% use  to prove energy conservation  for the extended solution. %We then find a condition on our original solution that ensures the required condition for $u_E$.
 %Finally we will relate back to the original solution to give a condition for energy conservation on $D_+$.
 %To this end we

%A detailed proof of this theorem on $\t^3$ and $\R^3$ is in \cite{RRS1} and is also shown in \cite{shvydkoy2010lectures} and \cite{RS09}. %Checked. \todo{proof in $D$? More explicit?}
%(May have to prove this here/argue why proofs in $\R^3,\t^3$ combined give .)

\subsection{Weak solutions of the Euler equations on $D_+:=\t^2\times \R_+$}

We define $\S(D_+)$ and $\S_\sigma(D_+)$ by restricting functions in $\S(\T^2\times\R)$ and $\S_\sigma(\T^2\times\R)$ to $D_+$; this means that we have Schwartz-like decay in the unbounded direction, and that the functions have a smooth restriction to the boundary.

We define $H_\sigma(D_+)$ to be the completion of $\D_\sigma(D_+)$ in the norm of $L^2(D_+)$; this is equivalent to the completion of
\begin{equation*}
\mathcal{S}_{n,\sigma}(D_+):= \{\phi\in \mathcal{S}(D_+):\nabla\cdot \phi =0 \text{ and } \phi_3=0 \text{ on } \partial D_+\}
\end{equation*}
in the same norm. Functions in $H_\sigma(D_+)$ are weakly divergence free in that they satisfy
\begin{equation}\label{dfH1}
\<u,\nabla\phi\>=0\qquad\mbox{for every}\quad\phi\in H^1(D_+);
\end{equation}
that this holds for every $\phi\in H^1(D_+)$ and not only for $\phi\in\D(D_+)$ (proved exactly as in Section \ref{sec:WSE1}) will be useful in what follows. %(The construction  and properties of $H_\sigma$ on a bounded set in $\R^3$ can be found in Section 2.2 of \cite{RRS}; for the domain $D_+$ the construction is very similar using $\mathcal{S}_{0,\sigma}(D_+)$ rather than $C^\infty_{0,\sigma}$ to ensure decay in the unbounded direction.)

As before,
%
%Then we can define the space $H_\sigma(D_+)$ as
%\begin{equation}\label{Hspace}
% H_\sigma(D_+):= \text{the completion of }  \mathcal{S}_{0,\sigma}(D_+) \text{ in the norm of } L^2(D_+).
%\end{equation}
%Then
in a slight abuse of notation we denote by $C_\w([0,T];H_\sigma (D_+))$ the collection of all functions $u\:[0,T]\to H_\sigma(D_+)$ that are weakly continuous into $L^2(D_+)$ i.e.
\begin{equation}\label{whatCwis}
 t\mapsto \langle u(t), \phi \rangle_{D_+}
\end{equation}
is continuous for every $\phi\in L^2(D_+)$.

% The space of functions  $C_\w([0,T];H_\sigma (D_+))$ is incompressible in the sense of distributions for almost every $t\in [0,T]$ and satisfies $u\cdot n=0$ on $\partial D_+$ in the sense of the Gauss formula: for $w\in H^{1/2}(\partial D_+)$, $\tilde w \in H^{1}( D_+)$ and for $u\in H_\sigma(D_+)$ we have the functional
%\begin{equation}
% F_u(w)=\int_{\partial D_+} w (u\cdot n) \d S_x:=\int_{D_+} (\nabla \cdot u )\tilde w\d x +\int_{D_+} u\cdot \nabla\tilde  w \d x =0.
%\end{equation}
%%\todo{More here? Explicit Gauss formula?}

 %the continuous dual space of $C^\infty(\partial D _+ \times [0,T])$.

We define %the space of test functions
\begin{equation*}
\S_\sigma^T(D_+):= \{\psi\in C^\infty(D_+\times[0,T]):\ \psi(\cdot,t)\in\S_\sigma(D_+)\mbox{ for every }t\in [0,T] \},
\end{equation*}
which will be our space of test functions; note that these functions are smooth and incompressible, but there is no restriction on their values on $\partial D_+$.
%where $\mathcal{S}$ denotes Schwartz functions.

To obtain a weak formulation of the equations on $D_+$ we consider first a smooth solution $u$ with pressure $p$ that satisfies the Euler equations
\begin{equation*}
\begin{cases}
 \partial_t u+ \nabla \cdot (u\otimes u)+ \nabla p =0 & \text{ in } D_+  \\
 \nabla \cdot u=0 & \text{ in } D_+ \\
 u\cdot n=0 & \text{ on } \partial D_+,
\end{cases}
\end{equation*}
where $n$ is the normal to $\partial D_+$, so that the third equation is in fact $u_3=0$ on $\partial D_+$. We can now multiply the first line by a test function $\phi\in \mathcal{S}_\sigma^T$ and integrate over space and time to give
\begin{equation*}
\int^t_0 \langle \partial_t u+ \nabla \cdot (u\otimes u)+ \nabla p, \phi\rangle_{D_+} \d \tau=0.
\end{equation*}
We can now integrate by parts and obtain
  \begin{multline*}
   \langle  u(t), \phi(t)\rangle_{D_+}-\langle  u(0), \phi(0)\rangle_{D_+}-\int^t_0 \langle  u,\partial_t \phi\rangle_{D_+} \d \tau- \int^t_0\langle u\otimes u:\nabla\phi \rangle_{D_+}\d \tau\\-\langle u_3, u\cdot\phi \rangle_{\partial D_+\times [0,t]}\d \tau - \int^t_0 \langle  p,\nabla \cdot  \phi\rangle_{D_+} \d \tau+\langle  p,  \phi\cdot n\rangle_{\partial D_+\times[0,t]} =0.
  \end{multline*}
We notice that as $u_3=0$ on $\partial D_+$  and $\nabla \cdot \phi =0$ in $D_+$ the two terms involving these expressions vanish and we have
\begin{equation*}
   \langle  u(t), \phi(t)\rangle_{D_+}-\langle  u(0), \phi(0)\rangle_{D_+}-\int^t_0 \langle  u,\partial_t \phi\rangle_{D_+} \d \tau- \int^t_0\langle u\otimes u:\nabla\phi \rangle_{D_+}\d \tau+\langle  p,  \phi\cdot n\rangle_{\partial D_+\times[0,t]} =0.
  \end{equation*}

%This is  for the boundary pressure term $p$.
Since we have not restricted the values of $\phi$ on $\partial D_+$ we have a contribution from the boundary, namely
\begin{equation*}
\langle  p,  \phi_3\rangle_{\partial D_+\times[0,t]}.
\end{equation*}
We therefore require $p\in \mathcal{D}'(\partial D _+ \times [0,T])$ in our definition of a weak solution.
% Thus, the weakest assumption on $p$ for this term to be defined, is for $p\in\mathcal{D}'(\partial D _+ \times [0,T])$.

\begin{definition}[Weak Solution on $D_+$]
A weak solution of the Euler equations on $D_+\times [0,T]$  is a pair $(u,p)$, where $u\in C_\w([0,T];H_\sigma (D_+))$ and $p\in\mathcal{D}'(\partial D _+ \times [0,T])$
such that
%(\, for\, n\, smooth\, )
\begin{multline}\label{WeaksolutionD+}
 \langle u(t),\phi(t)\rangle_{D_+}-\langle u(0),\phi(0)\rangle_{D_+}-\int^t_0\langle u(\tau),\partial_t\phi(\tau)\rangle_{D_+} \d \tau\\=\int^t_0\langle u(\tau)\otimes u(\tau):\nabla\phi(\tau)\rangle_{D_+}\d \tau-\langle  p,  \phi\cdot n\rangle_{\partial D_+\times[0,t]},
\end{multline}
for every $t\in[0,T]$ and for every $\phi\in \S_\sigma^T(D_+)$.
\end{definition}

Note that in the final term, $\phi\cdot n=-\phi_3$.
% \begin{remark}
% Here we have used the weakest assumptions possible and only assumed that $p$ is a distribution on the boundary as this is all we will need.% for this method.  	
% \end{remark}

\subsection{Half plane reflection map}

We introduce an extension $u_E$ that takes a weak solution $u$ defined in $D_+$ to one defined on the whole of $D$.
%this will be and odd extension of $u$ in the $x_3,3^{\mathrm{rd}}$ direction where we have the boundary. \eqref{Hspace}
Essentially we extend `by reflection', with appropriate sign changes to ensure that $u_R$, the `reflection' of $u$, is a weak solution on $D_-:=\t^2\times \R_-$. We can then show that $u_E:=u+u_R$ is a weak solution on the whole of $D$ (in the sense of Definition \ref{def:weakonU3}).
%by will have the useful properties of being incompressible and have $u\cdot n=0$  (Sense of Lemma 2.12 from \cite{RRS} where we define $H$) on the boundary

Given a vector-valued function $f\:D_\pm\to\R^3$ we define $f_R\:D_\mp\to\R^3$ by
\begin{equation*}
 f_R(x,y,z):= \begin{pmatrix}
       f_1(x,y,-z)\\
          f_2(x,y,-z)\\
          -f_3(x,y,-z)\\
         \end{pmatrix}
\end{equation*}
extending $f$ and $f_R$ by zero beyond their natural domain of definition, we  set 
$$
f_E(x,y,z):=\begin{cases}f(x,y,z)+f_R(x,y,z)&z\neq0\\
\frac{1}{2}(f(x,y,z)+f_R(x,y,z))=(f_1(x,y,0),f_2(x,y,0),0)&z=0.
\end{cases}
$$
Clearly $f_E=f+f_R$ almost everywhere.

\begin{lemma}
 If $u\in H_\sigma(D_+)$ then $u_R\in H_\sigma(D_-)$ and $u_E\in H_\sigma(D)$.
\end{lemma}

\begin{proof}
  The only claim that requires proof is that $u_E$ remains weakly divergence-free, despite possible issues near $x_3=0$. However, given any $\phi\in\D(D)$ we can write $\phi=\phi_++\phi_-$, where $\phi_\pm:=\phi|_{D_\pm}\in H^1(D_\pm)$; we can therefore use (\ref{dfH1}) to write
  $$
  \<u_E,\nabla\phi\>=\<u,\nabla\phi_+\>+\<u_R,\nabla\phi_-\>=0
  $$
  and $u_E$ is weakly divergence-free as claimed.
\end{proof}

Now we will show that, with an appropriate choice of the pressure, $u_R$ is a weak solution of the Euler equations in the lower half space $D_-$. Note that we do not need to extend the pressure distribution $p$.

%For $u$ a strong solution with $p$ the scalar pressure we have the Euler equations
% \begin{equation}
%  \partial_t u+ (u\cdot \nabla )u +\nabla p=0,
% \end{equation}
% with $\nabla \cdot u=0$ in $D_+$  and $u\cdot n=0$ on $\partial D_+$. When we look at the pair $(u_R,p_R)$ and use a simple change of variables of $(x_1,x_2,x_3)\to (y_1,y_2,-y_3)$, to map from $D_-$ to $D_+$, we find that this pair solves the Euler equations on the lower half plane.

\begin{theorem}\label{Lowerweaksolution}
 If $(u,p)$ is a weak solution to the Euler equations on $D_+$ then $ (u_R,p)$ is a weak  solution in $D_-$, i.e.
 \begin{multline}\label{WeaksolutionD-}
 \langle u_R(t),\phi(t)\rangle_{D_-}-\langle u_R(0),\phi(0)\rangle_{D_-}-\int^t_0\langle u_R(\tau),\partial_t\phi(\tau)\rangle_{D_-} \d \tau\\=\int^t_0\langle u_R(\tau)\otimes u_R(\tau):\nabla\phi(\tau)\rangle_{D_-}\d \tau-\langle  p,  \phi\cdot n\rangle_{\partial D_-\times[0,t]},
\end{multline}
for every $t\in[0,T]$ and for every $\phi\in \S_\sigma^T(D_-)$.
\end{theorem}

Note that now in the final term we have $\phi\cdot n=\phi_3$.

\begin{proof} Notice first that any $\phi\in\S_\sigma^T(D_-)$ can be written as $\psi_R$, where $\psi=\phi_R\in\S_\sigma^T(D_+)$. Now, the change of variables $(x_1,x_2,x_3)\to (y_1,y_2,-y_3)$ in the linear term yields
 \begin{equation*}
  \langle u_R,\psi_R \rangle_{D_-}=\langle u,\psi \rangle_{D_+}.
 \end{equation*}
For the nonlinear term one can check case-by-case, with the same change of variables, that
\begin{equation*}
 \int_{D_-}[(u_R)_i(u_R)_j \partial_j(\psi_R)_i](x) \d x=\int_{D_+}[u_iu_j \partial_j\psi_i](y) \d y.
\end{equation*}
Finally for the pressure term we have
$$
\<p,\psi\cdot n\>_{\partial D_+}=\<p,\psi_3\>=-\<p,\phi_3\>=\<p,\phi\cdot n\>_{\partial D_-},
$$
since $\psi_3(x,y,0)=-\phi_3(x,y,0)$.\end{proof}

%Since $u$ and $u_R$ are weak solutions on $D_+$ and $D_-$, respectively,   we now have two equations: the first
%\begin{multline}\label{Weakupperp}
% \langle u(t),\phi(t)\rangle_{D_+}-\langle u(0),\phi(0)\rangle_{D_+}-\int^t_0\langle u(\tau),\partial_t\phi(\tau)\rangle_{D_+} \d \tau=\\\int^t_0\langle u(\tau)\otimes u(\tau):\nabla\phi(\tau)\rangle_{D_+}\d \tau+ \langle  p,  \phi_3\rangle_{\partial D_+\times[0,t]},
%\end{multline}
%from the definition of a weak solution,  and  the second
%\begin{multline}\label{Weaklowerp}
% \langle  u_R(t),\eta  (t)\rangle_{D_-}-\langle   u_R(0),\eta(0)\rangle_{D_-}-\int^t_0\langle  u_R(\tau),\partial_t\eta(\tau)\rangle_{D_-} \d \tau=\\\int^t_0\langle  u_R(\tau)\otimes u_R(\tau):\nabla\eta(\tau)\rangle_{D_-}\d \tau-\langle  p,  \eta_3\rangle_{\partial D_+\times[0,t]},
%\end{multline}
% derived in Theorem \ref{Lowerweaksolution}  using our reflection map, where $\eta \in \mathcal{S}_\sigma(D_-\times [0,T])$.

By adding \eqref{WeaksolutionD+} and \eqref{WeaksolutionD-} it follows that $u_E$ is a weak solution of $D$.

 \begin{corollary}
 	The extension  $u_E$ is a weak solution of the Euler equations on $D$  in the sense of Definition \ref{def:weakonU3}.
\end{corollary} %\todo {Decomposition proof below}

\begin{proof}
For $\zeta \in \mathcal{S}_\sigma^T$ we can use $\zeta|_{D_+}$ as a test function in \eqref{WeaksolutionD+} and $\zeta|_{D_-}$ in \eqref{WeaksolutionD-} and add the two equations to obtain
%\begin{multline}\label{Rsolution2cts}
% \langle u_E(t),\zeta (t)\rangle_{D}-\langle u_E(0),\zeta(0)\rangle_{D}-\int^t_0\langle u_E(\tau),\partial_\tau\zeta(\tau)\rangle_{D} \d \tau=\\\int^t_0\langle u_E(\tau)\otimes u_E(\tau):\nabla\zeta(\tau)\rangle_{D}\d \tau -\int^t_0\int_{\{x_3=0\}} p (\zeta_3-\zeta_3) \d S_x \d t.
% \end{multline}
% The last term vanishes and we have
 \begin{align*}
 \langle u_E(t),\zeta (t)\rangle_{D}-\langle u_E(0),\zeta(0)\rangle_{D}-\int^t_0\langle u_E(\tau)&,\partial_\tau\zeta(\tau)\rangle_{D} \d \tau\\
 &=\int^t_0\langle u_E(\tau)\otimes u_E(\tau):\nabla\zeta(\tau)\rangle_{D}\d \tau,
 \end{align*}
 where the pressure terms have cancelled due to the opposite signs of the normal in the two domains; but this is now the definition of a weak solution of the Euler equations in $D$.
 %We can split $\zeta$ into $\zeta=\psi+\eta$ where $\psi \in \mathcal{S}_\sigma(D_+\times [0,T])$ and $\eta \in \mathcal{S}_\sigma(D_-\times [0,T])$ and we have $\eta_3-  \psi_3=0$ on $\{z=0\}$ so continuity at the boundary. As the supports of
\end{proof}

%  \begin{remark}
%   This proof can be generalised further and still works with just continuity across the boundary and so for $\zeta \in \mathcal{S}_\sigma(D_+\times [0,T])\cap \mathcal{S}_\sigma(D_-\times [0,T])\cap C(D\times [0,T])$ as all we needed was for $\zeta$ to have the same boundary value. In fact all we really need is for our test function to be continuous in the third component on the boundary and have the relevant regularity in the upper and lower half planes.
%  \end{remark}

 Since $u_E$ is a weak solution of the incompressible Euler equations on $D$,   Corollary \ref{conservation} guarantees that if $u_E\in L^3(0,T;L^3(D))$ and
\begin{equation}\label{conditionE}
 \lim_{|y|\to 0} \frac{1}{|y|}\int_{0}^{T}\int_{D}|u_E(t,x+y)-u_E(t,x)|^3\d x\d t=0
\end{equation}
  then $u_E$ conserves energy on $D\times [t_1,t_2]$. Due to the definition of $u_E$ this implies that
%
%
%  and applying Corollary \ref{conservation} we have energy conservation on $D\times[t_1,t_2]$.
%  We note that  due to the reflection map
 \begin{equation*}
 \|u_E(t_2)\|^2_{ L^2(D)}-\| u_E(t_1)\|^2_{ L^2(D)}=2\| u(t_2)\|^2_{ L^2(D_+)}-2\|u(t_1)\|^2_{L^2(D_+)}=0,
 \end{equation*}
 i.e.\ we obtain energy conservation for $u$. We now find conditions on $u$ alone (rather than $u_E=u+u_R$) that guarantee that \eqref{conditionE} is satisfied.

\section{Energy Conservation on $D_+$} %(Continuity at boundary)}\todo{New section assuming less}
%\todo{Continuity at the boundary extension}

% We will find that the main two interior conditions on $u$ that guaranties that the condition \eqref{conditionE} is satisfied by $u_E$ are similar.
Here we will prove our main result in Theorem \ref{MainD+}: energy conservation on $D_+$ under certain assumptions on the weak solution $u$.
The two bulk conditions we need for $u$ to  conserve energy are similar to the conditions needed for Corollary \ref{conservation} where we had no boundary.
We will impose two extra conditions to deal with the presence of the boundary: firstly, that there exists a $\delta>0$ such that $u\in{L^3(0,T;L^\infty(\t^2\times [0,\delta)}))$;
%This gives suitable integrability in time to use the Dominated Convergence Theorem to bring the limit inside the time integral.
secondly that $u(\cdot, t)$ is continuous at the boundary for almost every $t$. %to show that the boundary term disappears in the limit.

We make some preliminary definitions and observations concerning the kind of continuity we require at $\partial D_+$.

%\todo{Notation OK?}
\begin{definition}[Continuity at a Subset]
 We say that a function $f$ defined on $\Omega$ is $C_{\Gamma}$, for $\Gamma \subset \Omega$, if for all $x\in \Gamma$ and for each $\eps>0$ there exists a $\delta>0$ such that
 \begin{equation*}
  y\in\Omega\ \mbox{ and }\ |y-x|<\delta\qquad\Rightarrow\qquad |f(x)-f(y)|<\eps.
 \end{equation*}
\end{definition}
If $\Gamma$ is a compact subset of $\Omega$ then $f\in C_{\Gamma}$ is in fact uniformly continuous at the subset, in the following sense. % as expected we obtain uniform continuity at the boundary.
\begin{lemma}\label{uniforctsat}
 If $f$ is $C_{\Gamma}$ and $\Gamma$ is compact   then for all $\eps>0$ there exist $\delta>0$ such that for all $x\in \Gamma$
 \begin{equation*}
 y\in\Omega\ \mbox{ and }\ |y-x|<\delta\qquad\Rightarrow\qquad  |f(y)-f(x)|<\eps;
 \end{equation*}
 in particular, there exists a function $w\colon [0,\infty )\to [0,\infty )$ with $w(0)=0$ and continuous at $0$, such that
  \begin{equation*}
|f(x+z)-f(x)|<w(|z|)
 \end{equation*}
 whenever $x\in\Gamma$ and $x+z\in\Omega$.
 \end{lemma}

\begin{proof}
%We have a continuous function on a compact set, as a closed bounded subset of $\R^n$, and so we obtain a uniform continuity at the boundary which is what we want.
%
 Fix $\eps>0$ and take a sequence $y_n\in \Omega$ and $x_n\in \Gamma$ and assume that $|y_n-z_n|<\tfrac{1}{n}$ but $|f(y_n)-f(x_n)|\ge \eps$. However, we know that $\Gamma$ is compact and so there exists subsequences $y_{n_j}\to x$ and $x_{n_j}\to x$; by applying continuity at a subset for $f$ we have $f(y_{n_j})\to f(x)$ and   $f(x_{n_j})\to f(x)$, a contradiction.
\end{proof}

We can now provide conditions on $u$ to ensure energy conservation.

% \begin{corollary}
%  If $f$ is $C^@(\partial \Omega)$  with compact boundary $\partial \Omega$  then for all $x\in \partial \Omega$ let $w:[0,\infty)\to [0,\infty)$ be continuous at $0$ with $w(0)=0$ then for every $y\in  \Omega$
%  \begin{equation}
% |f(x+y)-f(x)|<w(|y|).
%  \end{equation}
%  \end{corollary}
% \begin{proof}
%  This the just the modulus of continuity applied to  Lemma \ref{uniforctsat}.
% \end{proof}

\begin{theorem}\label{MainD+}
Let $u\in L^3(0,T;L^3(D_+))$ be a weak solution of the Euler equations that satisfies $u\in L^3(0,T;L^\infty(\t^2 \times [0,\delta))$ for some $\delta>0$, $u(\cdot, t)\in C_{\partial D_+}$ for almost every $t$, and
\begin{equation}\label{conditiononu}
\lim_{|y|\to 0} \frac{1}{|y|}\int^{t_2}_{t_1}\iint_{\t^2}\int^\infty_{|y|}|u(t,x+y)-u(t,x)|^3\d x_3\d x_1\d x_2\d t=0;
\end{equation}
then $u$ conserves energy  on $D_+\times[t_1,t_2]$.
\end{theorem}

\begin{proof}
We can split
$$
 \lim_{|y|\to 0} \frac{1}{|y|}\int_{0}^{T}\int_{D}|u_E(t,x+y)-u_E(t,x)|^3\d x\d t=0
 $$
up into three sub-integrals over the regions $A:=\{x|x_3>|y|\}$, $B:=\{x|x_3<-|y|\}$ and $C:=\{x| |x_3|\le |y|\}$.
 We have
\begin{align*}
 |u_E(t,x+y)-u_E(t,x)|^3 & \le\biggl( [\mathbb{I}_{A}(x)+\mathbb{I}_{B}(x)+\mathbb{I}_{C}(x)]|u_E(x+y) -u_E(x)|\biggr)^3\\
&= \left[\mathbb{I}_{A}(x)+\mathbb{I}_{B}(x)+\mathbb{I}_{C}(x)\right]\,|u_E(x+y) -u_E(x)|^3.
\end{align*}
For $\int_A$  we see that since $x_3>0$ and $x_3+y_3>0$ then $u_E$ is in fact $u$, thus after integrating and taking the limit it goes to zero by \eqref{conditiononu}. For $\int_B$ a very similar argument (but with the change of variables $x_3\mapsto -z_3$) gives the same outcome.

 We are left with  $\int_C$:  we need to show that
\begin{equation*}
 \lim_{|y|\to 0} \frac{1}{|y|}\int^{t_2}_{t_1}\iint_{\t^2}\int_{-|y|}^{|y|}|u_E(t,x+y)-u_E(t,x)|^3\d x_3\d x_2\d x_1\d t=0.
\end{equation*}
We have assumed that $u\in L^3(0,T;L^\infty(\t^2 \times [0,\delta))$ and so
% as the minus sign of the $u_3(x,y,-z)$ component will have no effect on this norm, we have that
\begin{equation*}
 u_E\in L^3(0,T;L^\infty(\t^2 \times (-\delta,\delta)).
\end{equation*}
 % and the other components are just reflected normally.
 Then, since for all $|y|<\delta$ we have
\begin{equation*}
 \frac{1}{|y|}\iint_{\t^2}\int_{-|y|}^{|y|}|u_E(t,x+y)-u_E(t,x)|^3\d x_3\d x_2\d x_1\le C \sup_{x\in\t^2\times [0,\delta)} |u(t)|^3,
\end{equation*}
% then with $u_E$ in $L^3(0,T;L^\infty(\t^2 \times (-\delta,\delta))$ % the assumption of $L^3$  regularity in $t$
% we have that $\sup_x|u(t)|^3$
% is a bounded integrable function. % in times and so applying the  allows us to bring the limit inside the time integral.
 we can then  move the limit inside the time integral using the Dominated Convergence Theorem, and it suffices to show that
 \begin{equation*}
 \lim_{|y|\to 0} \frac{1}{|y|}\iint_{\t^2}\int_{-|y|}^{|y|}|u_E(t,x+y)-u_E(t,x)|^3\d x_3\d x_2\d x_1=0
\end{equation*}
for almost every $t\in(t_1,t_2)$.

As, $u(\cdot, t)\in C_{\partial D_+}$ and $u\cdot n=0$ on the boundary,  the boundary values are the same for $u$ and $u_R$ and so $u_E(\cdot,t) \in C_{\{z=0\}}$.% From this condition we see that , as using the condition $u\cdot n=0$, .

Now fix $t$ and let $x'\in \{z=0\}$; then   %\todo{May need more time regularity here?}
\begin{multline*}
 |u_E(t,x'+z+y)-u_E(t,x'+z)|\le |u_E(t,x'+z+y)-u_E(t,x')+u_E(t,x')-u_E(t,x'+z)|\\\le w(t,|y+z|)+w(t,|z|)\le 2w(t,2|y|)
\end{multline*}
and thus
\begin{multline*}
 \frac{1}{|y|}\iint_{\t^2}\int_{-|y|}^{|y|}|u_E(t,x+y)-u_E(t,x)|^3\d x_3\d x_2\d x_1\le C\frac{1}{|y|}\iint_{\t^2}\int_{-|y|}^{|y|}|w(t,2|y|)|^3\d x_3\d x_2\d x_1\\ \le C \frac{1}{|y|} |\t^2||y||w(t,2|y|)|^3 \to 0
\end{multline*}
as $|y|\to 0$, which is what we required.
\end{proof}

% \begin{remark}
%  The boundary conditions, $u\in L^3(0,T;L^\infty(\t^2 \times [0,\delta))$ for some $\delta>0$ and $u(\cdot, t)\in C^@(\partial D_+)$ for almost every $t$, were the weakest conditions that we needed around the boundary. These conditions give suitable control for the boundary to match up with the interior solution and for there to be no large jumps in the solution as you approach the boundary.
%  \end{remark}

 We note that the continuity and boundary assumptions required in the theorem  could be combined into
 %$L^3(0,T; C^@(\t^2 \times \{0\})$ or
 $L^3(0,T; C( \t^2 \times [0,\delta))$ for some $\delta>0$, or $L^3(0,T; C^\alpha( \t^2 \times [0,\delta))$ for some $\alpha>0$.
%\begin{remark}
 In fact all the conditions for this theorem are satisfied by a weak solution $u$ that satisfies
 \begin{equation*}
  |u(x,t)-u(y,t)|\le Cf(x_3)|x-y|^{\alpha}
 \end{equation*}
for $\alpha >\tfrac{1}{3}$ and $f\in L^3(0,\infty)$.

 %\in L^3(0,T; C^\gamma (\t^2\times \R_+))$ for $\gamma>\tfrac{1}{3}$.
 %. satisfying Onsager's conjecture on this domain with a boundary.
%\end{remark}

\section{Conclusion}

Assuming the simple integral condition
\begin{equation*}
\lim_{|y|\to 0} \frac{1}{|y|}\int^{t_2}_{t_1}\iint_{\t^2}\int^\infty_{|y|}|u(t,x+y)-u(t,x)|^3\d x_3\d x_1\d x_2\d t=0,
\end{equation*}
which is similar to the weakest condition known on $\R^3$ or $\t^3$, and appropriate continuity at the boundary we have proved energy conservation of the incompressible Euler equations with a flat boundary of finite area. Further, our methods do not depend on the dimension so analogues  hold in $\t^{d-1}\times \R_+ $ for $d\ge 2$.

\section*{Acknowledgements}

JLR is currently supported by the European Research COuncil, grant no.\ 616797. JWDS is supported by EPSRC as part of the MASDOC DTC at the University of Warwick, Grant No. EP/HO23364/1.

\end{document}